\documentclass[12pt,a4paper,reqno]{amsart}
\usepackage{amsmath}
\usepackage{amsfonts}
\usepackage{amssymb}
\numberwithin{equation}{section}
\usepackage{setspace}

\theoremstyle{plain}

\newtheorem{theorem}[subsection]{Theorem}

\newtheorem{lemma}[subsection]{Lemma}
\newtheorem{corollary}[subsection]{Corollary}

\theoremstyle{definition}

\newtheorem{definition}[subsection]{Definition}

\renewcommand{\leq}{\leqslant}
\renewcommand{\geq}{\geqslant}

\newsavebox{\proofbox}
\savebox{\proofbox}{\begin{picture}(7,7)%
  \put(0,0){\framebox(7,7){}}\end{picture}}

\def\E{\mathbb{E}}
\def\N{\mathbb{N}}
\def\F{\mathbb{F}}
\def\Z{\mathbb{Z}}
\def\R{\mathbb{R}}

\def\C{\mathbb{C}}

\def\B{\mathcal{B}}

\def\rank{\operatorname{rank}}
\def\eps{\varepsilon}

\def\f{{\mathbf f}}

\onehalfspace

\begin{document}

\title[Length 4 progressions in finite field geometries]{New bounds for Szemer\'edi's theorem, Ia: Progressions of length 4 in finite field geometries revisited}

\author{Ben Green}
\address{Centre for Mathematical Sciences, Wilberforce Rd, Cambridge CB3 0WA, England}
\email{b.j.green@dpmms.cam.ac.uk}

\author{Terence Tao}
\address{Department of Mathematics, UCLA, Los Angeles CA 90095-1555, USA.
}
\email{tao@math.ucla.edu}

\thanks{BJG acknowledges the generous support of the European Research Council through ERC-2011-StG grant number 279438. TT is supported by NSF grant DMS-0649473.}

\begin{abstract} Let $p \geq 5$ be a prime.  We show that the largest subset of $\F_p^n$ with no 4-term arithmetic progressions has cardinality $O_p(N(\log N)^{-2^{-22}})$, where $N := |\F_p|^n = p^n$.  A result of this type was claimed in our previous paper, but the proof had a gap (and we issue an erratum for that paper here).  We give here a different and significantly shorter argument that yields the same bound.  In fact we prove a stronger result, which can be viewed as a quantatitive version of some previous results of Bergelson-Host-Kra and the authors.
 \end{abstract}

\maketitle

\section{Introduction}

Szemer\'edi's theorem \cite{szemeredi} asserts that any set of integers with positive upper density contains arbitrarily long arithmetic progressions. This is easily seen to be equivalent to the assertion that $r_k(N) = o_k(N)$ for all $k \geq 3$,
where $r_k(N)$ denotes the cardinality of the largest subset of $[N] = \{1,\dots,N\}$ containing no $k$-term arithmetic progression with distinct terms, and $o_k(N)$ denotes a quantity which, when divided by $N$, goes to zero as $N \to \infty$ for each fixed $k$.

Much attention has been devoted to the question of finding bounds for $r_k(N)$. The current state of the art is as follows:
\begin{enumerate}
\item Sanders \cite{sanders-r3} showed in 2010 that $r_3(N) \ll N(\log N)^{-1 + o(1)}$;
\item The authors \cite{green-tao-szem2} showed in 2005 that $r_4(N) \ll N e^{-c\sqrt{\log \log N}}$;
\item Gowers \cite{gowers-long-aps} showed in 1998 that $r_k(N) \ll N(\log \log N)^{-c_k}$ for every $k \geq 5$.
\end{enumerate}
We omit a detailed discussion of the history of the problem, referring the reader to the three papers cited above.

In studying these problems a great deal of mileage has been gained from studying what are known as \emph{finite field models}. Instead of $r_k(N)$ one considers $r_k(F^n)$, where $F$ is a finite field. The quantity $r_k(F^n)$ is defined to be the cardinality of the largest subset of the vector space $F^n$ containing no $k$-term arithmetic progression with distinct terms. In order that a $k$-term arithmetic progression not be degenerate, we must assume that $F$ has characteristic greater than $k$, and we assume that $F = \F_p$ is a prime field for notational simplicity. When $k = 3$ one traditionally takes $F = \F_3$, and for the purposes of this paper, where our main interest lies in the case $k = 4$, the reader will lose little by taking $F = \F_5$.  See \cite{green-fin-field} for a general discussion of the role of finite field models in additive combinatorics. 

Write $N := |F^n|$. Then the current state of the art for this question is as follows:
\begin{enumerate}
\item Bateman and Katz \cite{bateman-katz} showed in 2011 that $r_3(F^n) \ll_F N(\log N)^{-1 - c}$ for some absolute constant $c > 0$;
\item The authors \cite{gt-inverseu3} showed in 2005 that $r_4(F^n) \ll N(\log \log N)^{-c_F}$;
\item The authors \cite{green-tao-r4-finite-fields} in 2009 improved this bound to $r_4(F^n) \ll_F N e^{-c_F\sqrt{\log\log N}}$.  We also claimed the improved bound $r_4(F^n) \ll_F N (\log N)^{-c_F}$.
\item It is known, for instance by using the density Hales-Jewett theorem \cite{FK}, that $r_k(F^n) = o_{k,F}(N)$ for all $k \geq 5$, assuming of course that $F$ has characteristic at least $k$. 
\end{enumerate}

Recently, we discovered that our argument in \cite{green-tao-r4-finite-fields} claiming the bound $r_4(F^n) \ll_F N (\log N)^{-c_F}$ contains a gap, the nature of which is described in Appendix \ref{erratum}.  (The ``cheap'' bound $r_4(F^n) \ll_F N e^{-c_F\sqrt{\log\log N}}$ established in that paper is however not subject to this problem, nor is the analogous bound for $r_4(N)$ established in \cite{green-tao-szem2} by similar methods.) The main purpose of this paper is to provide an alternate, simpler, and -- most importantly -- correct argument that recovers this bound.  In fact, we obtain the following stronger statement. By an \emph{affine subspace} of $F^n$ we mean a coset of a linear subspace $\dot W$ of $F^n$.

\begin{theorem}\label{main}
Let $F = \F_p$ be a finite field with $p \geq 5$. Let $n \in \N$, let $0 < \alpha, \eps \leq 1$, and let $A$ be a subset of $F^n$ of density at least $\alpha$.  Then there exists an affine subspace $W$ of $F^n$ of codimension at most $C_{F} \eps^{-2^{20}}$ with the property that
\begin{equation*}
\label{goal}
 |\{ (x,r) \in W \times \dot W: x,x+r,x+2r,x+3r \in A \}| \geq (\alpha^4-\eps) |W|^2,
\end{equation*}
where $C_F > 0$ depends only on $F$.
\end{theorem}

A qualitative variant of this theorem already appeared (as a joint result of the authors of the present paper) in \cite[Theorem 4.1]{montreal}, which in turn was inspired by an ergodic theoretic result of Bergelson, Host, and Kra \cite{BHK}; see also \cite[Theorem 1.12]{green-tao-reg} for another related result.  Note that the quantity $\alpha^4 |W|^2$ is the natural quantity associated to the statistic $|\{ (x,r) \in W \times \dot W: x,x+r,x+2r,x+3r \in A \}|$, as if $A$ were a random subset of $F^n$ with density $\alpha$, then the expected value of this statistic would indeed be $\alpha^4 |W|^2$.  The exponent $2^{20}$ is certainly not best possible, and is mostly dependent on the exponent $2^{16}$ appearing in the inverse theorem for the $U^3$ norm in \cite{gt-inverseu3}; any improvement on the exponents in the latter result would lead to improvements in the exponents here.

As an immediate corollary of the above theorem, we recover the main result claimed in \cite{green-tao-r4-finite-fields}.

\begin{corollary}\label{main-cor}
Let $F = \F_p$ be a finite field with $p \geq 5$. Let $n \in \N$, and write $N := |F^n|$.  Then $r_4(F^n) \ll_F N(\log N)^{-2^{-22}}$.
\end{corollary}

\begin{proof}  Let $A$ be a subset of $F^n$ with no length $4$ progressions and cardinality $r_4(F^n)$, and set $\alpha := |A|/|F^n| = r_4(F^n)/N$.  By Theorem \ref{main} with $\eps=\alpha^4/2$ (say), we can find an affine subspace $W$ of $F^n$ of codimension at most $C_F \alpha^{-2^{22}}$ for some $C_F>0$ depending only on $F$, such that
$$ |\{ (x,r) \in W \times \dot W: x,x+r,x+2r,x+3r \in A \}| \geq \frac{1}{2} \alpha^4 |W|^2.$$
On the other hand, as $A$ has no length $4$ progressions, the left-hand side is at most $|W|$. We conclude that $|W| \leq 2/\alpha^4$ which, when combined with the codimension bound on $W$, implies that 
$$n \leq C_F \alpha^{-2^{22}} + \log_{|F|} \frac{2}{\alpha^4} \ll_F \alpha^{-2^{22}}.$$
This gives $\alpha \gg_F (\log N)^{-2^{-22}}$, and the claim follows.\vspace{11pt}
\end{proof}


\section{Notation and an outline of the argument}\label{outline-sec}
 
Throughout this paper the field $F$ is fixed, and all constants are permitted to depend on $F$.  As such we will no longer explicitly subscript these constants by $F$, for instance abbreviating $c_F$ as $c$.

For technical reasons it is convenient to replace the vector space $F^n$ by the more general concept of an \emph{affine space}, by which we mean a coset $W = x + \dot W$ of a linear subspace $\dot W$ of some ambient vector space $F^n$, where $x$ is also an element of $F^n$.  We will often refer to $W$ without any explicit mention of the underlying space $F^n$. The dimension of $W$, $\dim(W)$, is defined to be $\dim(\dot{W})$.  If $W'$ is an affine space which is contained in another affine space $W$, we call $W'$ an \emph{affine subspace} of $W$, and define the \emph{codimension} of $W'$ inside $W$ to be $\dim(W) - \dim(W')$.

Our argument is similar to that in \cite{montreal} or \cite{green-tao-reg}, but with more attention paid to the quantitative estimates.  The main step in our argument will be what we call a \emph{local Koopman-von Neumann theorem}, the detailed statement of which is Theorem \ref{kvn-theorem}.  Roughly speaking, this theorem asserts that if $A$ is a subset of some affine space $W$ of some density $\alpha$, then we can find an an affine subspace $W'$ of $W$ of large codimension on which $A$ can be approximated (in the sense of the Gowers $U^3(W')$ norm) by a ``quadratically structured'' function $f$, that is to say a function of a bounded number of quadratic polynomials on $W'$.   Furthermore we may ensure that the density of $A$ on $W'$ is basically at least as large as $\alpha$, and crucially we may also ensure that the quadratic polynomials involved in the construction of $f$ obey a ``high rank'' condition, in the sense that any non-trivial linear combination of these polynomials has high rank.  The most important ingredient in the proof of Theorem \ref{kvn-theorem} is the inverse theorem for the Gowers $U^3$-norm in finite fields \cite{gt-inverseu3}.  

Once Theorem \ref{kvn-theorem} is proven it follows from the theory of Gowers norms that the count of $4$-term arithmetic progressions of $A$ in $W'$ is very close to the corresponding count of $4$-term arithmetic progressions weighted by $f$.  On the other hand, by invoking a ``counting lemma'' we will be able to obtain an accurate and explicit Fourier-analytic formula for the number of $4$-term arithmetic progressions weighted by $f$.  It turns out that there is a useful positivity property in this formula, essentially first observed in \cite{BHK} in a slightly different context, which allows 
one to give a lower bound for this count of essentially $\alpha^4$. This gives the main theorem.

The paper is organised as follows. In \S \ref{lambda-sec} we define the Gowers $U^3$-norm and prove some simple facts relating it to 4-term progressions. Section \S \ref{kvn-sec} is the heart of the paper: here we prove the local Koopman-von Neumann theorem, Theorem \ref{kvn-theorem}. Section \S \ref{mixing-sec} is concerned with analysing quadratically structured functions, and in particular with counting 4-term progressions weighted by them. From this, the main theorem is easily established.

\emph{Notation.} Our notation is standard in additive combinatorics. We draw the reader's attention to our use of $\E_{x \in X} f(x)$ to denote the average of $f$ over the (finite) set $X$. We write $\Vert f \Vert_{L^1(X)} := \E_{x \in X} |f(x)|$ and $\Vert f \Vert_{L^2(X)} := (\E_{x \in X} |f(x)|^2)^{1/2}$. We use the letter $C$ to denote an absolute constant; it need not be the same at every occurrence. When we want to emphasise different constants we use subscripts and refer to $C_0, C_1, C_2,\dots$. In this paper, each constant $C$ could be specified explicitly if desired.  We use $X \ll Y$ or $X=O(Y)$ to denote the bound $|X| \leq C Y$ for some constant $C$. 
 
\section{Progression of length 4 and the $U^3$ norm}\label{lambda-sec}

Recall from the previous section the notion of an affine space $W$ with associated linear space $\dot{W}$.

 Let $W$ be an affine space over $F$. If $f_0,f_1,f_2,f_3 : W \rightarrow \R$ are functions then we define
 \[ T_W(f_0,f_1,f_2,f_3) := \E_{x \in W,h \in \dot{W}} f_0(x) f_1(x+d) f_2(x + 2d) f_3(x + 3d),\] a normalised count of the 4-term arithmetic progressions in $W$ weighted by the functions $f_0,f_1,f_2$ and $f_3$.
 In the special case in which all the $f_i$ are equal to some function $f$ then we will  write
 \[ T_W(f) := T_W(f,f,f,f).\]
 
We record a bound for $T_W$ in terms of the Gowers $U^3$-norm, a result of a type known as a \emph{generalized von Neumann theorem}. For a much lengthier introduction to the Gowers $U^3$-norm, see \cite{gt-inverseu3}. If $f : W \rightarrow \C$ is a function, we define $\Vert f \Vert_{U^3(W)}$ to be the unique non-negative real number such that
\begin{eqnarray*} \Vert f \Vert_{U^3(W)}^8 & := & \E_{x \in W; h_1,h_2,h_3 \in \dot W}(f(x)\overline{f(x+h_1)f(x+h_2)f(x+h_3)}f(x+h_1+h_2) \times \\
& & \qquad\qquad\qquad\qquad \times f(x+h_2+h_3)f(x+h_1+h_3) \overline{f(x+h_1+h_2+h_3)}).\end{eqnarray*}
This is the standard definition, modified slightly so that it applies to affine spaces as well as linear ones. It can be shown that the quantity on the right is real and non-negative, so $\Vert f \Vert_{U^3(W)}$ is well-defined. It can also be shown that $\Vert  \cdot \Vert_{U^3(W)}$ defines a norm, but we shall not need this fact in this paper.

The next lemma is of a type referred to in the literature as a Generalised Von Neumann Theorem.

\begin{lemma}\label{gvn-lem} Let $W$ be an affine space and suppose that $f_0,f_1,f_2,f_3 : W \rightarrow \C$ are bounded in magnitude by $1$. Then we have
\[ |T_W(f_0,f_1,f_2,f_3)| \leq \min_{0 \leq i \leq 3} \|f_i\|_{U^3(W)}.\]
\end{lemma}
 \begin{proof} This is \cite[Proposition 1.7]{gt-inverseu3}, and is proved in \S 4 of that paper using three applications of the Cauchy-Schwarz inequality. Versions of this inequality appear in several earlier works also, such as \cite{gowers-long-aps}.  The extension to affine spaces $W$ is trivial and left to the reader.\end{proof}
 
Using the telescoping identity 
\[T_W(f)-T_W(g) = T_W(f-g,g,g,g)  + T_W(f,f-g,g,g)  + T_W(f,f,f-g,g)  + T_W(f,f,f,f-g),
\]
we conclude the following bound.

\begin{lemma}\label{u3-lambda4} 
Let $f, g: W \to \C$ be functions on an affine space $W$ bounded in magnitude by $1$.  Then we have
$$ |T_W(f) - T_W(g)| \leq 4 \|f-g\|_{U^3(W)}.$$
\end{lemma}

\section{Factors and Quadratically structured functions}\label{kvn-sec}

In this section we develop the language and tools needed to discuss the ``quadratically structured functions'' mentioned in \S \ref{outline-sec}.

\begin{definition}[Factors] If $W$ is a finite set then by a \emph{factor} $\B$ we mean simply a partition of $W$ into finitely many pieces which, in this paper, we refer to as \emph{atoms}.
\end{definition}
\emph{Remark.}  The nomenclature hints at connections with ergodic theory which in some sense inspire some of the arguments of this paper. We say that a function $\phi : W \rightarrow \C$ is $\B$-measurable if it is constant on atoms of $\B$. 

If $f : W \rightarrow \C$ is any function then we may define the \emph{conditional expectation} $$ \E(f|\B)(x) := \E_{\B(x)}f \hbox{ for all } x \in W,$$
where $\B(x)$ is the unique atom in $\B$ that contains $x$.  Equivalently, $\E(f|\B)$ is the orthogonal projection (in the Hilbert space
$L^2(W)$) to the space $\B$-measurable functions. 

Suppose that we are given a finite collection $\phi_1,\dots,\phi_d$ of functions from $X$ to some other set $Y$. Then these may be used to define a factor $\B = \B_{\phi_1,\dots, \phi_d}$ in a natural way by taking the atoms of $\B$ to consist of sets of the form $\{x \in X : \phi_1(x) = y_1, \dots, = \phi_d(x) = y_d\}$. For factors defined in this way we refer to $d$ as (an upper bound for) the \emph{complexity} of the factor $\B$.

We say that a factor $\B'$ is a \emph{refinement} of $\B$ if every atom in $\B$ is a union of atoms in $\B'$.
We will also need the notion of the \emph{join} $\B \vee \B'$ of two factors, which is simply the factor formed by intersecting the atoms of $\B$ with those of $\B'$ (or equivalently, the minimal factor that refines both $\B$ and $\B'$). Note that $\B_{\phi_1,\dots, \phi_d} \vee \B_{\phi'_1,\dots, \phi'_{d'}} = \B_{\phi_1,\dots, \phi_d, \phi'_1,\dots, \phi_{d'}}$ for any functions $\phi_i, \phi'_i$.

\begin{definition}[Quadratic functions]
Suppose that $W$ is a linear space. By choosing a basis for $W$ we may identify it with $F^n$ for some $n$. By a \emph{quadratic  function} on $W$ we mean a function $\phi : W \rightarrow F$ of the form $\phi(x) = x^T M x + r^T x + c$, where $M$ is an $n \times n$ symmetric matrix over $F$, $r \in F^n$, and $c \in F$.  By the \emph{rank} of $\phi$ we understand the rank of the matrix $M$.
More generally, if $W = \dot{W} + w$ is an affine space then $\phi : W \rightarrow F$ is a quadratic function if the function $\dot{\phi} : \dot{W} \rightarrow F$ defined by $\dot{\phi}(x) := \phi(x + w)$ is a quadratic function on $\dot{W}$.  We define the rank of $\phi$ to be the rank of $\dot{\phi}$.\end{definition}

\begin{definition}[Quadratic factor]
If $X = W$ is an affine space and the $\phi_i$ are all quadratic functions then we refer to $\B = \B_{\phi_1,\dots,\phi_d}$ as a \emph{quadratic factor}. 
\end{definition}

We will be mostly interested in quadratic factors with a particularly pleasant property.

\begin{definition}[Quadratic factors and rank]\label{quad-factor-def}
Let $W$ be an affine space. Then by a \emph{quadratic factor of rank at least $r$ and complexity $d$} we mean a factor $\B = \B_{\phi_1,\dots,\phi_d}$ defined by quadratic functions $\phi_1,\dots,\phi_d : W \rightarrow F$ which satisfy the rank separation condition $\rank(\lambda_1 \phi_1 + \dots + \lambda_d \phi_d) \geq r$ whenever $\lambda_1,\dots,\lambda_d$ are elements of $F$, not all zero.
\end{definition}

The utility of the rank separation condition will become clear as we proceed, and is particularly clearly illustrated by Lemma \ref{quad-expect}, where it is shown that all atoms of $\B$ have roughly the same size if one assumes this condition. One may also count arithmetic progressions across atoms of a high-rank quadratic factor: see Lemma \ref{quad-expect-4}.  Some related use of high rank quadratic factors and functions occur in \cite{gowers-wolf,green-tao-reg,tz-low}.
\vspace{11pt}

For technical reasons we will need to ``localise'' quadratic factors to certain subspaces.  This requires some additional definitions.

\begin{definition}[Local factors]  Let $W$ be an affine space.  By a \emph{local factor of codimension at most $D$} we mean a factor $\B_1$ of $W$ whose atoms are all affine subspaces of $W$ of codimension at most $D$; note that we allow these subspaces to have different orientations (and even different codimensions).  By a \emph{local quadratic factor of codimension at most $D$, rank at least $r$, and complexity at most $d$} we mean a pair $(\B_1,\B_2)$ of factors, where $\B_1$ is a local factor on $W$ of codimension at least $D$, and $\B_2$ is an extension of $\B_1$ with the property that on each atom $W'$ of $\B_1$, the restriction $\B_2\downharpoonright_{W'}$ of $\B_2$ to $W'$ is a quadratic factor on $W$ of rank at most $r$ and complexity at most $d$.  

We say that a local quadratic factor $(\B'_1,\B'_2)$ is a refinement of another local quadratic factor $(\B_1,\B_2)$ if $\B'_1$ is a refinement of $\B_1$ and $\B'_2$ is a refinement of $\B_2$.
\end{definition}

\emph{Some facts about factors.} In this subsection we collect together some lemmas about factors, and quadratic factors in particular.

\begin{lemma}\label{refine}
Suppose that $X$ is a finite set and that $f : X \rightarrow \C$ is a function. Suppose that $\B$ and $\B'$ are two factors, with $\B'$ a refinement of $\B$. Then
\[ \Vert \E(f | \B') \Vert_{L^2(X)} \geq \Vert \E(f | \B) \Vert_{L^2(X)}.\]
\end{lemma}
\begin{proof}
We have $\E( \E(f | \B') | \B) = \E(f | \B)$, and so $\E(f | \B)$ is the orthogonal projection of $\E(f | \B')$ (in $L^2(X)$) to the space of $\B$-measurable functions. In particular $\E(f | \B') - \E(f | \B)$ is orthogonal to $\E(f | \B)$, and so Pythagoras' theorem yields
\[ \Vert \E(f | \B') \Vert_{L^2(X)}^2 = \Vert \E(f | \B) \Vert_{L^2(X)}^2 + \Vert \E(f | \B') - \E(f | \B) \Vert_{L^2(X)}^2 \geq \Vert \E(f | \B) \Vert_{L^2(X)}^2.\]This concludes the proof.\vspace{11pt}
\end{proof}

We shall refer to $\Vert \E(f | \B) \Vert_{L^2(X)}^2$ as the \emph{energy} of $f$ relative to the factor $\B$. Note that if $f$ is bounded (by 1) then the energy lies in the interval $[0,1]$. 

The following lemma, which shows how to make a quadratic factor high-rank, is crucial.

\begin{lemma}\label{rank-reduce}
Suppose that $W$ is an affine space and that $\B$ is a quadratic factor of complexity at most $d$ on $W$. Then there is a local quadratic factor $(\B_1,\B_2)$ of codimension at most $dr+d^2+d$, rank at least $r$, and complexity at most $d$, such that $\B_2$ is a refinement of $\B$.
\end{lemma}

\begin{proof} 
Suppose that $\B$ is defined by quadratic forms $\phi_1,\dots, \phi_d$. If, for every choice of $\lambda_1,\dots,\lambda_d \in \F$, not all zero, we have the high-rank condition $\rank(\lambda_1 \phi_1 + \dots + \lambda_d \phi_d) \geq r+d$, then the result is immediate (with $\B' := \B$). Otherwise, we may rescale and relabel so that, without loss of generality, $\lambda_d = 1$. Consider the homogeneous linear space $\dot{W}$. The fact that this rank is at most $r$ means that the kernel of $\lambda_1 \dot{\phi}_1 + \dots + \lambda_d \dot{\phi}_d$, $\dot{W}'$ say, has codimension at most $r$. Restricted to this kernel, $\dot{\phi}_d$ is a linear combination of $\dot{\phi}_1,\dots,\dot{\phi}_{d-1}$. 

If now $\rank(\lambda_1 \dot{\phi}_1 + \dots + \lambda_{d-1} \dot{\phi}_{d-1}) \geq r+d$ then stop; otherwise, continue this rank reduction process. It clearly lasts at most $d$ steps, at which point (after relabelling) we have a subspace $\dot{W}' \leq \dot{W}$ of codimension at most $d(r+d)$ and some $d'$, $0 \leq d' \leq d$, such that, restricted to $\dot{W}'$, each of $\dot{\phi}_1,\dots, \dot{\phi}_d$ is a linear combination of $\dot{\phi}_1,\dots,\dot{\phi}_{d'}$.

This means that, restricted to any coset $V$ of $\dot{W}'$ in $W$, the factor $\B\downharpoonright_V$ has as a refinement a factor cut out by the $d'$ quadratics $\phi_1,\dots,\phi_{d'}$, which satisfy a rank condition with parameter $r+d$, as well as up to $d$ linear phases. The affine subspaces cut out by these linear phases, over all cosets $V$, then forms a local factor $\B_1$ of codimension at most $d(r+d)+d$.

Restricted to an atom $W'$ of the local factor $\B_1$, the quadratics $\phi_1,\dots,\phi_{d'}$ still satisfy a rank condition with parameter $r$. Take $\B_2 := \B \vee \B_1$, then $(\B_1,\B_2)$ is a local quadratic factor of codimension at most $dr+d^2+d$, rank at least $r$, and complexity at most $d$ as desired.
\end{proof}

We have studied the properties of quadratic factors, but we have yet to say why they are useful. The next result, an inverse theorem for the $U^3(W)$-norm, is the key input in this regard. Here, $e_F : F \rightarrow \C$ is defined by $e_F(x) = e^{2\pi i x/p}$, where $F = \F_p$ is identified with $\Z/p\Z$.

\begin{theorem}\label{invert-u3} Let $W$ be a linear space over $F$, and let $f: W \to \C$ be a bounded function such that $\|f\|_{U^3(W)} \geq \eta$ for some $0 < \eta \leq \frac{1}{2}$.  Then there is a linear subspace $W' \leq W$ of codimension at most $O(\eta^{-2^{16}})$	
such that, for each coset $W' + t$ of $W'$ in $W$, there exists a quadratic phase function $\phi_{t}: W' + t \to F$ such that
\begin{equation}\label{tww}
|\E_{t \in W/W'} |\E_{x \in W'} f(x) e_F(- \phi_t(x) )| \gg \eta^{2^{16}}.
\end{equation}
\end{theorem}
\begin{proof} See \cite[Theorem 2.3]{gt-inverseu3}.\vspace{11pt}\end{proof}

We have the following corollary of this in the language of factors.

\begin{corollary}[Inverse theorem for $U^3$, corollary] \label{u3-reform}
Let $W$ be an affine space and suppose that $f : W \rightarrow \C$ is a bounded function such that $\Vert f \Vert_{U^3(W)} \geq \eta$, where $0 < \eta \leq \frac{1}{2}$. Then there is a local quadratic factor $(\B_1,\B_2)$ of codimension $O(\eta^{-2^{16}})$ and complexity at most $1$ such that $\| \E(f|\B_2)\|_{L^2(W)} \gg \eta^{2^{16}}$.
\end{corollary}

\begin{proof}  Without loss of generality we may take $W$ to be a linear space.
Let $W'$ and the $\phi_t$ be as in Theorem \ref{invert-u3}, and let $\B_1$ be the local factor generated by the cosets of $W'$, 
thus $\B_1$ has codimension $O(\eta^{-2^{16}})$.  Let $\B_2$ be the factor whose atoms are of the form $\{ x \in W'+t: \phi_t(x) = a \}$ for various $t \in W/W'$ and $a \in F$: then $(\B_1,\B_2)$ is a local quadratic factor of codimension $O(\eta^{-2^{16}})$ and complexity at most $1$.  Observe that the left-hand side of \eqref{tww} can be rewritten as
$$ |\E_{t \in W/W'} |\E_{x \in W'} \E(f|\B_2)(x) e_F(- \phi_t(x) )|,$$
which, by the Cauchy-Schwarz inequality, is bounded by $\|\E(f|\B_2)\|_{L^2(W)}$. The claim follows.
\end{proof}

\begin{theorem}[Local Koopman-von Neumann]\label{kvn-theorem}
Let $A \subseteq W$ be a set with density $\alpha$, $0 < \alpha \leq 1$, on some affine space $W$. Let $0 < \eta,\eps < \frac{1}{2}$, and suppose that $r \geq 1$. Then there is an affine subspace $W' \subseteq W$ of codimension $O(\eps^{-3} \eta^{-2^{19}} r)$ such that the density of $A$ on $W'$ is at least $\alpha - \eps$, and such that there is a quadratic factor $\B$ on $W'$ of rank at least $r$ and complexity $O(\eps^{-1} \eta^{-2^{17}})$ such that $\Vert 1_{A} - \E(1_{A}|\B) \Vert_{U^3(W')} \leq \eta$.
\end{theorem}

\begin{proof} 
For $i = 0, 1, 2\dots$ we are going to define a local quadratic factor $(\B_{1,i},\B_{2,i})$ on $W$ of codimension at most $d_i$, rank at least $r$, and complexity at most $i$. To initialise the construction, we set $\B_{1,0}$ and $\B_{2,0}$ to be the trivial factor $\{\emptyset, W\}$ on $W$.
Suppose we have completed this construction up to and including step $i$. Consider an atom $W'$ of $\B_{1,i}$, thus $W'$ is a subspace of codimension at most $d_i$.  Let us say that such an atom $W'$ is \emph{regular} if $\Vert 1_{A} - \E(1_{A} | \B_i) \Vert_{U^3(W')} \leq \eta$. If the union of the regular atoms of $\B_{1,i}$ has density less than $1 - \eps/2$ in $W$ then we continue to step $(i+1)$; otherwise we stop. 

If an atom $W'$ of $\B_{1,i}$ is not regular, then by Corollary \ref{u3-reform} we may find a local quadratic factor $(\B_{1,i,W'}, \B_{2,i,W'})$ on $W'$ of codimension $O( \eta^{-2^{16}})$ and complexity at most $1$ (with no bound on the rank at present), such that

\begin{equation}\label{eq70} \| \E( 1_A - \E(1_A|\B_{2,i}) | \B_{2,i,W'}) \|_{L^2(W')} \gg \eta^{2^{16}}.\end{equation}

If $W'$ is regular, we set $\B_{i,W'}$ to be the trivial factor $\{\emptyset,W'\}$ on $W'$.

For $j=1,2$, let $\B'_{j,i}$ be the factor generated by $\B_{j,i}$ and each of the $\B_{j,i,W'}$ as $W'$ varies over the atoms of $\B_{1,i}$, thus the restriction of $\B'_{j,i}$ to each atom $W'$ of $\B_{1,i}$ is simply $\B_{j,i}\downharpoonright_{W'} \vee \B_{j,i,W_i}$.  Then $(\B'_{1,i}, \B'_{2,i})$ is a local quadratic factor of codimension at most $d_i + O(\eta^{-2^{16}})$ and complexity at most $i+1$, which is a refinement of $(\B_{1,i}, \B_{2,i})$.  
The rank properties of the original local quadratic factor $(\B_{1,i}, \B_{2,i})$ have been destroyed by the passage to the extension $(\B'_{1,i}, \B'_{2,i})$, but we can recover the rank property using Lemma \ref{rank-reduce}.  Namely, if $W''$ is an atom of $\B'_{1,i}$, then by applying Lemma \ref{rank-reduce} to the quadratic factor $\B'_{2,i} \downharpoonright_{W''}$, we may find a local quadratic factor $(\B''_{1,i,W''}, \B''_{2,i,W''})$ on $W''$ of  codimension at most $(i+1)r + (i+1)^2 + i+1$, rank at least $r$, and complexity at most $i+1$, with $\B''_{2,i,W''}$ refining $\B'_{2,i} \downharpoonright_{W''}$.  Gluing together the $(\B''_{1,i,W''}, \B''_{2,i,W''})$ as $W''$ varies among the atoms of $\B'_{1,i}$, we obtain a local quadratic factor $(\B_{1,i+1},\B_{2,i+1})$ of codimension at most $d_{i+1} := d_i + O( \eta^{-2^{16}}) + (i+1)r + (i+1)^2 + i+1$, complexity at most $i+1$, and rank at least $r$, which refines $(\B'_{1,i},\B'_{2,i})$ and hence $(\B'_{i},\B'_{i})$.

From \eqref{eq70} and Lema \ref{refine} we have
$$
\| \E( 1_A - \E(1_A|\B_{2,i}) | \B_{2,i+1}) \|_{L^2(W')}^2 \gg \eta^{2^{17}}$$
for each irregular atom $W'$ of $\B_{1,i}$.  For regular atoms $W'$ we use the trivial lower bound of $0$.  Averaging over all atoms $W'$ we conclude that
$$
\| \E( 1_A - \E(1_A|\B_{2,i}) | \B_{2,i+1}) \|_{L^2(W)}^2 \gg \eps \eta^{2^{17}}.
$$
By Pythagoras' theorem, the left-hand side can be rewritten as $\|\E(1_A|\B_{2,i+1})\|_{L^2(W)}^2 - \|\E(1_A|\B_{2,i})\|_{L^2(W)}^2$, and so we have the energy increment
$$ \|\E(1_A|\B_{2,i+1})\|_{L^2(W)}^2 \geq \|\E(1_A|\B_{2,i})\|_{L^2(W)}^2 + c \eps \eta^{2^{17}}$$
for some constant $c = c_F >0$.
On the other hand, the energy $\|\E(1_A|\B_{2,i})\|_{L^2(W)}^2$ clearly can only take values between $0$ and $1$, and therefore this iteration can only occur at most $O( \eps^{-1} \eta^{-2^{17}})$ times.  At each stage of the iteration, the complexity of the factor increases by at most one, and the codimension increases by at most
$$ O( \eta^{-2^{16}}) + (i+1)r + (i+1)^2 + i+1 \ll \eps^{-2} \eta^{-2^{18}} r$$
since $i = O( \eps^{-1} \eta^{-2^{17}})$.  At the end of this iteration, we obtain a final local quadratic factor $(\B_{1,i},\B_{2,i})$ of codimension $O( \eps^{-3} \eta^{-2^{19}} r )$, rank at least $r$, and complexity $O( \eps^{-1} \eta^{-2^{17}} )$, with the property 
\begin{equation}\label{wij}
\Vert 1_{A} - \E(1_{A} | \B_{2,i}) \Vert_{U^3(W')} < \eta
\end{equation}
for all atoms $W'$ of $\B_{1,i}$, outside of an exceptional set of atoms whose union has density at most $\eps/2$ in $W$.

As before, we call an atom $W'$ of $\B_{1,i}$ \emph{regular} if \eqref{wij} holds.  We wish to find a regular value $W'$ of $\B_{1,i}$ for which, in addition, the density of $A$ is at least $\alpha - \eps$. Suppose this is not possible. Then we have
$$\E_{W'} (1_A - \alpha) <  - \eps$$
for all regular $W'$, while for irregular $W'$ we have the trivial upper bound of $1$.  Averaging in $j$, we conclude that
$$ \E_W (1_A - \alpha) < -\eps (1-\eps/2) + \eps/2 < 0.$$
But the left-hand side is zero by definition of $\alpha$, a contradiction, and the claim follows.

If we now set $W'$ to be a regular atom of $\B_{1,i}$ on which $A$ has density at least $\alpha-\eps$, and $\B$ to be the restriction of $\B_{2,i}$ to $W'$, we obtain the conclusion of Theorem \ref{kvn-theorem}.
\end{proof}

\section{High-rank quadratic factors}\label{mixing-sec}

 We turn now to a more detailed study of quadratic factors of high rank, showing how to control the size of atoms in these factors, and later how to count 4-term arithmetic progressions in functions measurable with respect to one of these factors.
 
 Suppose that $W$ is a linear space, and that $\phi_1,\dots, \phi_d : W \rightarrow F$ are quadratic maps. Let $\B = \B_{\phi_1,\dots, \phi_d}$ be the quadratic factor defined by the $\phi_i$, that is to say the partition of $W$ in which the atoms are sets of the form $\{ x : \phi_1(x) = c_1, \dots, \phi_d(x) = c_d\}$. Throughout this section we will assume that $\B$ has rank at least $r$, which means that the homogeneous parts $\dot{\phi}_1,\dots, \dot{\phi}_d$ satisfy the rank separation condition $\rank(\lambda_1 \dot{\phi}_1 + \dots + \lambda_d \dot{\phi}_d) \geq r$ whenever $\lambda_1,\dots,\lambda_d \in F$ are not all zero.
 
An important role will be played by the map $\Phi : W \rightarrow F^d$ defined by $\Phi(x) = (\phi_1(x),\dots, \phi_d(x))$.
Note that an atom of $\B$ is simply the inverse image, in $W$, of some point in $F^d$ under this map $\Phi$. If $f : W \rightarrow \C$ is a bounded $\B$-measurable function then we write $\f : F^{d} \rightarrow \C$ for the function which satisfies $f(x) = \f(\Phi(x))$ for all $x \in W$. 

Suppose that $F = \F_p$, which we identify with $\Z/p\Z$. Write $e_F : F \rightarrow \C^{\times}$ for the standard character on $F$, which maps $x$ to $e(x/p)$ where $e(t) := e^{2\pi i t}$. Our first lemma is a standard Gauss sum estimate.

\begin{lemma}\label{gauss-lem}
Suppose that $W$ is an affine space and that $\phi : W \rightarrow F$ is a quadratic form with rank $r$. Then $|\E_{x \in W} e_F(\phi(x))| = |F|^{-r/2}$.
\end{lemma} 
 \begin{proof} By translating if necessary (which does not affect the rank) we may identify $W$ with $F^n$. Suppose that $\phi(x) = x^T M x + r^T + c$ with $M$ symmetric.
 
 Squaring and changing variables, we have
\[ |\E_{x \in F^n} e_F(\phi(x))|^2 = |\E_{x,h} e(\phi(x + h) - \phi(x))|  = |\E_{x,h} e_F(2h^T M x)|.\]
If $Mx \neq 0$ then the expectation over $h$ vanishes. If $Mx = 0$, which happens for $|F|^{n-r}$ values of $x$, then it equals $1$. Therefore $|\E_{x \in F^n} e_F(\phi(x))|^2 = |F|^{-r}$, which is the stated result.\vspace{11pt}\end{proof}

Using this lemma we can show that the atoms in a high-rank quadratic factor have roughly the same size. We phrase this as a result about averaging functions, as follows.

\begin{lemma}\label{quad-expect} Let 
$\B$ be a 
quadratic factor of complexity $d$ on an affine space $W$, with rank at least $r$.  Let $\Phi$ be the corresponding map from $W$ to $F^d$.
Let $f: W \to \C$ be a bounded $\B$-measurable function, and let $\f$ be the corresponding function on $F^d$.  Then $|\E_W(f) - \E_{F^d}(\f)| \leq |F|^{(d-r)/2}$.
\end{lemma}

\begin{proof} We employ a Fourier expansion on $F^d$. The dual of $F^d$ may be identified with $F^d$ itself by associating to $\xi \in F^d$ the character $x \mapsto e_F(\xi \cdot x)$. Thus we define the Fourier transform
\[ \widehat{\f}(\xi) := \E_{x \in F^d} \f (x) e_F(-\xi \cdot x).\]
By the inversion formula we have
\[ f(x) = \f(\Phi(x)) = \sum_{\xi \in \F^d} 
\widehat \f(\xi) e_F( \xi\cdot \Phi(x) ).\]
Since $\widehat \f(0) = \E_{F^d}(\f)$, we conclude that
\[ \E_W(f) - \E_{F^d}(\f) =
\sum_{\xi \in F^d \setminus \{0\}} 
\widehat \f(\xi) \E_{x \in W} e_F( \xi \cdot \Phi(x) ).\]
Now from the rank hypotheses we see that $\xi \cdot \Phi(x)$ is a quadratic phase of rank at least $r$ whenever $\xi \in F^d \setminus \{0\}$. Therefore, the expectation has magnitude at most $|F|^{-r/2}$ by Lemma \ref{gauss-lem}.
Thus by the triangle inequality we have
$$ |\E_W(f) - \E_{F^d}(\f)| \leq |F|^{-r/2}
\sum_{\xi \in F^d} |\widehat \f(\xi)|.$$
By Cauchy-Schwarz and Plancherel we have
$$ \sum_{\xi \in F^d} |\widehat \f(\xi)|
\leq |F|^{d/2} \|\f\|_{L^2(F^d)},$$
and the claim now follows from the boundedness of $\f$.\vspace{11pt}
\end{proof} 

 We turn now to the somewhat more complicated task of counting 4-term arithmetic progressions using the configuration space. It is easy to see that, for any $x \in W$ and $h \in \dot{W}$ we have the relation $\Phi(x) - 3\Phi(x+h) + 3 \Phi(x+2h) - \Phi(x+3h) = 0$.
It turns out that if the rank $r$ is sufficiently large then this is in some sense the ``only'' constraint on the points $\Phi(x + ih)$, and furthermore there is
a certain uniform distribution among all the values of $\Phi(x+ih)$ obeying this constraint.  This leads to
the heuristic formula
\[T_W(f) \approx \E_{x_0, x_1, x_2, x_3 \in F^d : x_0 - 3x_1 + 3x_2 - x_3=0}
\prod_{i=0}^3 \f( x_i ),\]
which can be rearranged using the Fourier transform as
\[T_W(f) \approx  \sum_{\xi \in F^{d}} |\hat{\f}(\xi)|^2 |\hat{\f}(3\xi)|^2.\]

The next lemma constitutes the rigorous version of the above heuristics.

\begin{lemma}\label{quad-expect-4}  Let 
$\B$ be a quadratic factor of complexity $d$ on an affine space $W$, with rank at least $r$.  Let $\Phi$ be the corresponding map from $W$ to $F^d$, and let $f(x) = \f(\Phi(x))$ be a bounded $\B$-measurable function.  Then we have
\[
|T_W(f) -  \sum_{\xi \in F^{d}}|\hat{\f}(\xi)|^2 |\hat{\f}(3\xi)|^2 | \leq |F|^{(4d-r)/2}.
\]
\end{lemma}
\begin{proof}  Once again we use the Fourier expansion
$$ f(x) = \sum_{ \xi \in F^{d}} \widehat \f(\xi) 
e(\xi\cdot \Phi(x))$$
to obtain
\begin{equation}\label{to-use1} T_W(f) =
\sum_{\xi_0,\xi_1,\xi_2,\xi_3 \in F^d} m(\xi_0,\xi_1,\xi_2,\xi_3) \prod_{i=0}^3 \widehat \f(\xi_i)\end{equation}
where 
\begin{equation}\label{to-use2}m(\xi_0,\xi_1,\xi_2,\xi_3) := \E_{x \in W,h \in \dot{W}} e( \sum_{i=0}^3  \xi_i \cdot \Phi(x+ih) ).\end{equation}
Write $\Sigma \in (F^{d})^4$ for the set of all 4-tuples $(\xi_0,\xi_1,\xi_2,\xi_3)$ such that
\begin{equation}\label{xi-constraints}
3\xi_1 = -\xi_2 = \xi_3 = -3\xi_4 .
\end{equation}
We will shortly show that, for all choices of the $\xi_i$,
\begin{equation}\label{sigma-err}
|m(\xi_0,\xi_1,\xi_2,\xi_3) - 1_{\Sigma}(\xi_0,\xi_1,\xi_2,\xi_3)| \leq |F|^{-r/2}.
\end{equation}
Assuming this, we can compare \eqref{to-use1} with 
\[ \sum_{\xi \in F^{d}}|\hat{\f}(\xi)|^2 |\hat{\f}(3\xi)|^2  = \sum_{\xi_1,\xi_2,\xi_3,\xi_4 \in F^d} 1_{\Sigma}(\xi_0,\xi_1,\xi_2,\xi_3)\prod_{i=0}^3 \hat{f}(\xi_i),\]
obtaining
$$|T_W(f) -  \sum_{\xi \in F^{d}}|\hat{\f}(\xi)|^2 |\hat{\f}(3\xi)|^2 |  \leq  |F|^{-r/2} \sum_{ \xi_0, \xi_1,\xi_2,\xi_3 \in F^{d}} \prod_{i=0}^3 |\widehat \f(\xi_i)|.$$
Applying Cauchy-Schwarz and Plancherel as in the proof of the preceding lemma, we can bound
this by $|F|^{(4d-r)/2}$ as desired. 

It remains to prove \eqref{sigma-err}.  If $(\xi_0,\xi_1,\xi_2,\xi_3) \in \Sigma$ then this is trivial, since $m(\xi_0,\xi_1,\xi_2,\xi_3) =1$ in this case. Suppose, then,  that we do not have \eqref{xi-constraints}.
Then (by a simple inspection) we can find $i' \in \{0,1,2,3\}$ such that
$\sum_{i=0}^3 (i-i') \xi_i \neq 0$.  We can use the change of variables $x = y - i'h$ to write
\[m(\xi_0,\xi_1,\xi_2,\xi_3) = \E_{y \in W,h \in \dot{W}} e( \sum_{i=0}^3 \xi_i \cdot \Phi(y+(i-i')h) ).\]
It then follows from the rank condition that the phase $\sum_{i=0}^3 \xi_i \cdot \Phi(y+(i-i')h)$ contains a non-trivial quadratic
component in $h$ of rank at least $r$. By averaging over $h$ and applying Lemma \ref{gauss-lem}, we see that $m(\xi_0,\xi_1,\xi_2,\xi_3)$ has magnitude at most $|F|^{-r/2}$.  This concludes the proof of \eqref{sigma-err} and hence of the lemma.\vspace{11pt}
\end{proof} 

We now take advantage of the pleasant positivity properties of the sum \[ \sum_{\xi \in F^{d}}|\hat{\f}(\xi)|^2 |\hat{\f}(3\xi)|^2\] appearing in the preceding lemma to conclude the following lower bound.

\begin{corollary}\label{quad-4-increment} Let $W$ be an affine space, and suppose that
$\B$ is a quadratic factor on $W$ with complexity at most $d$ and rank $r \geq 10d$.  Suppose that $A \subseteq W$ is a set of density at least $\alpha$.  Then $T_W(\E(1_A| \B))  \geq \alpha^4 - O( |F|^{-3d} )$.
\end{corollary}

\begin{proof}  Write $f := \E(1_A | \B)$ for notational brevity. Let $\Phi$ and $\f$ be as before: recall that $\Phi(x) = (\phi_1(x),\dots,\phi_{d}(x))$, where the $\phi_i$ are the quadratics defining $\B$ and that $\f$ is the unique $\B$-measurable function such that $\f(\Phi(x)) = f(x)$. Applying Lemma \ref{quad-expect-4}, and noting that $|F|^{(4d-r)/2} \leq |F|^{-3d}$, we have
$$ T_W(f) \geq \sum_{\xi \in F^{d}} |\hat{\f}(\xi)|^2 |\hat{\f}(3\xi)|^2 - |F|^{-3d}.$$
In particular, discarding all the terms with $\xi \neq 0$, we have
$$ T_W(f) \geq |\hat{\f}(0)|^4 - |F|^{-3d}.$$
Meanwhile, since $f$ has mean at least $\alpha$, we see from Lemma \ref{quad-expect} that
$$ |\hat{\f}(0)| \geq \alpha - |F|^{-3d}$$
(say).  The claim follows.\vspace{11pt}
\end{proof}

We can now prove Theorem \ref{main}.  Let $\alpha,\eps,A$ be as in that theorem.  We will weaken the conclusion of Theorem \ref{main} by replacing $\eps$ with $O(\eps)$; clearly, the original statement of the theorem can then be recovered by modifying $\eps$ by a multiplicative constant.  Thus, our objective is now to find an affine subspace $W'$ of $F^n$ of codimension $O( \eps^{-2^{20}} )$ such that $T_{W'}( 1_A ) \geq \alpha^4-O(\eps)$.
We may assume that $\eps \leq \alpha^4$, as the claim is trivial otherwise.

Set $\eta := \eps$, $d := \lfloor C_0 \eps^{-1} \eta^{-2^{17}}\rfloor = O( \eps^{-2^{18}} )$, and $r := 10d$ for some sufficiently large constant $C_0>0$ depending only on $F$.  By Theorem \ref{kvn-theorem}, we may find a subspace $W'$ of codimension $O( \eps^{-2^{20}} )$ and a quadratic factor $\B$ on $W'$ of rank at least $r$ and complexity at most $d$ such that $A$ has density at least $\alpha-\eps$ on $W'$, and such that
$$\Vert 1_{A} - \E(1_{A}|\B) \Vert_{U^3(W')} \leq \eps.$$
By Lemma \ref{u3-lambda4} it follows that
$$ T_{W'}( 1_A ) \geq T_{W'}( \E(1_A|\B) ) - O(\eps).$$
On the other hand, from Corollary \ref{quad-4-increment} one has
$$ T_{W'}( \E(1_A|\B) ) \geq (\alpha-\eps)^4 - O( |F|^{-3d}) = \alpha^4 - O(\eps) - O(|F|^{-3d}).$$
By choice of $d$, we certainly have $O(|F|^{-3d}) = O(\eps)$, and Theorem \ref{main} follows.

\appendix

\section{Erratum to previous paper}\label{erratum}

In this appendix we describe the error in our previous paper \cite{green-tao-r4-finite-fields}.

Fix some finite field $F$ of characteristic greater than $3$, for example $F = \F_5$. The main result \cite[Theorem 1.1]{green-tao-r4-finite-fields} of the aforementioned paper was a claimed proof of a statement of the same type as Corollary \ref{main-cor}: if $W$ is an affine space over $F$ and if $A \subseteq W$ has density at least $n^{-c}$, then $A$ contains four distinct elements in arithmetic progression. The attempted proof went via the so-called \emph{density increment strategy}: supposing that $A$ has density $\alpha$ and contains no 4-term progression, we located a reasonably large affine subspace $W' \leq W$ on which the density of $A$ is appreciably larger than $\alpha$. Iteration of this statement led to a contradiction.

This density increment was found in two steps. First of all the characteristic function $1_A$ was approximated in the Gowers $U^3$-norm by a ``quadratically structured'' function $\E(1_A | \B)$, where $\B$ is a quadratic factor: a partition of the underlying space $W$ into atoms defined by a collection of linear and quadratic phases. The relevant statement here is \cite[Theorem 6.6]{green-tao-r4-finite-fields} (a type of Koopman-Von Neumann theorem).

Secondly, we studied the number of 4-term progressions weighted by a quadratically structured function such as $\E(1_A | \B)$. A precise statement is \cite[Theorem 8.5]{green-tao-r4-finite-fields}. This eventually led to the conclusion that $A$ has increased density on some atom of $\B$, which we then decomposed into affine linear pieces to get the desired density increment.

This second phase required $\B$ to be \emph{high-rank}, which means that the quadratic phases defining $\B$ satisfy a rank separation condition (\cite[Definition 8.2]{green-tao-r4-finite-fields}, and see also Definition \ref{quad-factor-def} of the present paper). However, the factor $\B$ output by the Koopman-von Neumann theorem need not be high-rank. To get around this issue we stated and proved a lemma, \cite[Lemma 8.7]{green-tao-r4-finite-fields}, allowing one to refine an arbitrary quadratic factor $\B$ to a high-rank factor $\B'$. 

The problem with this is that, whilst $\E(1_A | \B)$ approximates $1_A$ in the $U^3$-norm, the same need not be true\footnote{As written in \cite{green-tao-r4-finite-fields}, this issue manifests itself in a slightly different way, namely in the last line of the paper when Theorem 8.8 is invoked in an attempt to prove Theorem 4.1.  Unfortunately, Theorem 8.8 is applied to a function $g = \E(f|\B_2)$ rather than to $f$ itself, and a density increment on $g$ on a subspace does not necessarily imply a corresponding density increment on $f$, because these subspaces do not come from partitioning an atom of $\B_2$, but rather from partitioning an atom from a finer factor $\B'$.  The obvious fix for this is to replace $g$ by $\E(f|\B')$, but this runs into the difficulty mentioned in the main text.} of $\E(1_A | \B')$.
What is needed is a Koopman-von Neumann theorem in which the output factor $\B$ is already high-rank. A result of this type is the main new development in this paper, specifically Theorem \ref{kvn-theorem}. Unfortunately we were only able to achieve this with usable bounds after first passing to a (large) subspace $W' \leq W$. We proceed using an energy-increment argument of basically the same type as that usually used to prove Koopman-von Neumann theorems, but with an additional rank-refinement step at each increment.

We remark that somewhat similar issues, albeit in a rather different language, are encountered (and correctly addressed) in \cite{gowers-wolf}. See in particular Theorem 5.7 there. In their application they cannot afford to pass to a subspace, and this is why their main theorem requires bounds of double-exponential type.

\end{document}